\documentclass{daj}

\usepackage{amsmath,amssymb,amsthm,bbm,mathtools,comment,cleveref}
\DeclareMathOperator*{\ourbox}{\text{\raisebox{-0.25ex}{\scalebox{1.25}{$\square$}}}}
\newcommand\Osq{\mathbin{\text{\scalebox{.84}{$\square$}}}}

\theoremstyle{plain}
\newtheorem{theorem}{Theorem}[section]		
\newtheorem{lemma}[theorem]{Lemma}

\newtheorem{corollary}[theorem]{Corollary}

\newtheorem{problem}[theorem]{Problem}

\theoremstyle{remark}

\dajAUTHORdetails{%
  title = {A Multidimensional Ramsey Theorem}, %% please capitalize all significant words
  author = {Ant\'onio Gir\~ao,
 Gal Kronenberg, and Alex Scott},
  plaintextauthor = {Antonio Girao,
 Gal Kronenberg, and Alex Scott},
  plaintexttitle = {A Multidimensional Ramsey Theorem}, %%  title without math or LaTeX
  runningtitle = {A Multidimensional Ramsey Theorem}, 
  runningauthor = {Ant\'onio Gir\~ao,
 Gal Kronenberg, and Alex Scott},
  copyrightauthor = {Ant\'onio Gir\~ao,
 Gal Kronenberg, and Alex Scott},
  keywords = {Ramsey Theory, Erd\H{o}s-Szekeres.},
}   %%% END \dajAUTHORdetails

\dajEDITORdetails{%
   year={2024},
   number={21},
   received={9 April 2023},   % received date: example: 7 January 2017
   published={31 December 2024},  % published date
   doi={10.19086/da.127777},       % XXX = number of paper, e.g. da006 for paper#6
}   %%% END \dajEDITORdetails

\begin{document}

\begin{frontmatter}[classification=text]
%% EDITOR: this will force the keywords to appear right after the Abstract.
%%   If the abstract is too long and would force the keywords off the
%%   front page, please comment out % [classification=text] above
%%   This way the keywords will be floated on the bottom of the first page
%%   even though the Abstract spills over to the next page.

%%% AUTHOR: Title goes here.  This line is optional.  You must use it
%%   if title has footnote attached or requires nontrivial typesetting,
%%   e.g., inclusion of linebreaks to force nice layout.
\title{A Multidimensional Ramsey Theorem} %% please capitalize all significant words

%%% AUTHOR:
%%% List all authors. If you wish, place grant acknowledgements in \thanks.
%%% In brackets include a short tag for each author.
\author[antonio]{Ant\'onio Gir\~ao \thanks{Mathematical Institute, University of Oxford, Andrew Wiles Building, Radcliffe Observatory Quarter, Woodstock Road, Oxford, United Kingdom.  E-mail: \texttt{\{girao,kronenberg,scott\}@maths.ox.ac.uk}.}\text{ } \footnotemark[3]}
\author[gal]{Gal Kronenberg  \footnotemark[1]\text{ } \thanks{Research supported by the European Union's Horizon
2020 research and innovation programme under the Marie Sk\l odowska Curie grant agreement No. 101030925, and by the Royal Commission for the Exhibition of 1851.}}
\author[alex]{Alex Scott  \footnotemark[1]\text{ } \thanks{Research supported by EPSRC grant EP/V007327/1. AS would like to thank the Department of Mathematics at the University of Arizona where some of this work was completed.}}

%%% AUTHOR: Abstract goes here
\begin{abstract}
In this paper, we prove a
``multidimensional" generalisation of Ramsey's Theorem to cartesian products
of graphs, showing that a doubly exponential upper bound is enough in
every dimension. 
More precisely, we prove that for every $r,n,d\in \mathbb{N}$, in any $r$-colouring of the edges of the Cartesian product $\ourbox^{d}\!K_N$ of $d$ copies of $K_N$ there is a copy of $\ourbox^{d}\!K_n$ such that the edges in each direction are monochromatic, provided $N\geq 2^{2^{C(d)rn^{d}}}$. 
  As an application of our new approach we obtain improvements on the multidimensional Erd\H{o}s-Szekeres Theorem proved by Fishburn and Graham 30 years ago thus confirming a conjecture posed by Buci\'c, Sudakov, Tran.
  \end{abstract}
\end{frontmatter}

%%% AUTHOR: body of paper starts here
\section{Introduction}

% One of the most famous areas in combinatorics is Ramsey theory. As usual, for integers $r,k$ and $n>k$, we denote $R_r(k)$ to be the smallest $n$ for which every $r$-colouring of $K_n$ contains a monochromatic copy of $K_k$. In 1930, Ramsey showed that this number always exists~\cite{Ramsey}. Since then, Ramsey numbers have been studied extensively, upper and lower bounds have been proved and many generalisations were considered[ some references]. Even though the exact value of $R_r(k)$ is still a mystery, it is known that this value is exponential in $k$. It  was already shown by Ramsey that $ R_2(k)\leq 2^{2k}$. A lower bound of $2^{k/2}$ was proved in the 40's by Erd\H{o}s~\cite{RamseyLowerBound} which was one of the first applications of the probabilistic method in Combinatorics. More generally, we know $2^{rk/2} \leq R_r(k) \leq r^{(r-1)(k-1)+1}$ and these bounds are essentially the state of art. We remark that even for $k=3$, we do not understand the behaviour of $R_r(3)$ as $r\rightarrow \infty$ and that very recently, Conlon and Ferber \cite{FerberConlon}, and Wigderson \cite{Wigderson}, found nice constructions which give exponential improvements on the lower bounds for $R_r(k)$ when $r\geq 3$. \gal{Should we cite more Ramsey history here?}

The study of Ramsey theory is a longstanding and central part of combinatorics.
As usual, for positive integers $r,k$, the Ramsey number $R_r(k)$ is the smallest $n$ for which every $r$-colouring of $K_n$ contains a monochromatic copy of $K_k$. Ramsey~\cite{Ramsey} showed in 1930 that these numbers exist. Since then, Ramsey numbers have been studied extensively, upper and lower bounds have been proved and many generalisations have been considered, see, e.g. \cite{ConlonRamseyUpper,ErdosRamseyLower,ErdosSzekeres,RamseyLowerBound,LefmannRamsey,SahRamsey,SpencerRamseyLower,ThomasonRamsey}.  Even for $r=2$, the asymptotics are still not fully resolved. It is not hard to show that $R_2(k)$ grows at exponential rate, but the constant in the exponent is yet not known: the best current bounds are $(1+o(1))\frac {\sqrt 2 (k+1)}e2^{k+1/2} \leq R_2(k+1) \leq e^{-c\log^2k}\binom {2k}{k}$ the lower bound due to Spencer \cite{SpencerRamseyLower} and upper bound by Sah \cite{SahRamsey}.  
For larger $r$, the gap between the upper and
lower bound is even larger, and even for $k=3$, we do not understand the behaviour of $R_r(3)$ as $r\rightarrow \infty$. Very recently, Conlon and Ferber \cite{FerberConlon}, Wigderson \cite{Wigderson}, and Sawin \cite{Sawin} found nice constructions which give the best known lower bounds for $R_r(k)$ when $r\geq 3$. 
In a recent breakthrough, Campos, Griffiths, Morris and Sahasrabudhe \cite{CamposRamsey2023}, reduced the the long standing upper bound of $4^k$ to $(4-\varepsilon)^k$. This is the first exponential improvement over the upper bound
of Erd\H{o}s and Szekeres, proved in 1935.
% For larger $r$, the bounds are less good, although it is known that 
% $2^{rk/4} \leq R_r(k) \leq r^{(r-1)(k-1)+1}$ [refs].  Even for $k=3$, we do not understand the behaviour of $R_r(3)$ as $r\rightarrow \infty$ and that very recently, Conlon and Ferber \cite{FerberConlon} and Wigderson \cite{Wigderson} found nice constructions which give exponential improvements on the lower bounds for $R_r(k)$ when $r\geq 3$. 

In this paper, we prove a ``multidimensional" generalisation of Ramsey's Theorem for cartesian products of graphs. 
Given two graphs $H$ and $G$, we write $G\Osq H$ for the Cartesian product of $H$ and $G$, namely the graph 
with vertex set $V(G)\times V(H)$ in which $(x,y)$ is joined to $(x',y')$ if and only if $x=x'$ and $yy'\in E(H)$ 
or $xx'\in E(G)$ and $y=y'$.  The Cartesian product is associative, so it makes sense to write $G_1\Osq G_2\Osq\cdots\Osq G_d$ (without brackets) for the Cartesian product of $d$ graphs; we write $\ourbox^{d}G$ 
for the product $G\Osq\cdots\Osq G$ of $d$ copies of $G$. Note that in a Cartesian product of $d$ graphs $G_1,\dots,G_d$, there is an edge between $v=(v_1,\dots,v_d)$ and $w=(w_1,\dots,w_d)$ if and only if there is some $i$ such that $v_iw_i$ is an edge of $G_i$ and $v_j=w_j$ for $j\ne i$, and in this case we will say that the edge $vw$ is in {\em direction $i$}.

Given a colouring $c$ of the edges of $\ourbox^{d}\!K_n$, we say that $c$ is {\em monochromatic in every direction} if for each $i\in\{1,\dots,d\}$ there is some $c_i$ such that all edges in direction $i$ have colour $c_i$.  
For positive integers $r,d$, we define 
$R_r(d,n)$ to be the smallest $N$ such that every $r$-colouring of the edges of $\ourbox^{d}\!K_N$ contains a copy of $\ourbox^d\!K_n$ that is monochromatic in every direction.  Note that we cannot demand a copy of $\ourbox^{d}\!K_n$ that has the {\em same} colour in every direction, as we are asking for a full-dimensional subgraph: for example, $\ourbox^{d}\!K_N$ could be coloured with colour $1$ for all edges in direction $1$ and  coloured with 2 for the edges in all other directions. It is easy to see that if we demand a monochromatic copy of $\ourbox^{\ell}\!K_n$ then $\ell$ must be at most $\lceil d/r\rceil$. It will follow from \Cref{thm:mainRamsey} that this is also tight.
%
%Let $r,d$ be positive integers, given an $r$-edge colouring of $B=\Osq^{(d)}K_n$, we say $B$ contains a %\textit{monochromatic} copy of a $\Osq^{(d)}K_k$ if for every dimension $i\in [d]$, there is a colour $c_i$ such that %all the $1$-dimensional $K_r$'s along dimension $i$ are monochromatic of colour $c_i$. We let $R_r(d,k)$ to be the %smallest $n$ such that every $r$-colouring of $\Osq^{(d)}K_n$ contains a monochromatic copy of $\Osq^{(d)}K_k$. 

It is not too hard to prove that $R_r(d,n)$ exists by an iterated application of Ramsey's Theorem. However, this would only give an upper bound of tower-type as a function of $d$. Our main goal is to show that a doubly exponential bound on $n^{d}$ suffices. 
\begin{theorem}\label{thm:mainRamsey}
%Let $d$ be a positive integer. There exists $C_d>0$ such that for every $n,r$ the following holds. Every $r$-colouring of $\ourbox^{d}\!K_N $ there is a copy of $\ourbox^{d}\!K_n$ which is monochromatic in every direction, provided $N\geq r^{r^{C_drn^{d}}}$. That is, $R_r(d,n)\leq r^{r^{C_drn^d}}$.
Let $d$ be a positive integer. There exists $C_d>0$ such that for every $n,r$ the following holds. 
For $N\geq r^{r^{C_drn^{d}}}$,
every $r$-colouring of $\ourbox^{d}\!K_N $ contains a copy of $\ourbox^{d}\!K_n$ which is monochromatic in every direction. That is, $R_r(d,n)\leq r^{r^{C_drn^d}}$.
\end{theorem}

As an immediate corollary, we see that any $r$-edge colouring of $\ourbox^{d}\!K_N$ contains a \textit{monochromatic} copy of $\ourbox^{\ell}\!K_n$ for $\ell=\lceil d/r\rceil $. 

\begin{corollary}
Let $n,d,r$ be positive integers and $\ell =\lceil d/r\rceil$. For $N\geq r^{r^{C_drn^{d}}}$, every $r$-edge-colouring of $\ourbox^{d}\!K_N$ contains a monochromatic copy of $\ourbox^{\ell}\!K_n$. The value of $\ell$ is tight. 
%Let $n,d,r$ be positive integers and $\ell =\lceil d/r\rceil$. Then, any $r$-edge-colouring of $\ourbox^{d}\!K_N$ contains a monochromatic $\ourbox^{\ell}\!K_n$, provided that $N\geq r^{r^{C_drn^{d}}}$. The value of $\ell$ is tight. 
\end{corollary}

% Our method also applies to monotone arrays, answering a question of Buci\'c, Sudakov, and Tran~\cite{Sudakov}.
Another foundational result in Ramsey theory appears in a paper  
%It is common to refer to the seminal paper of 
Erd\H{o}s and Szekeres~\cite{ErdosSzekeres} from 1935:
%. Their old and widely used result  states 
any sequence of $n^2+1$ distinct real numbers contains either an increasing
or decreasing subsequence of length $n+1$. 
%As a byproduct of our proof method, we prove a result on monotone arrays, thus answering a question of Buci\'c, Sudakov, and Tran  \cite{Sudakov}.
 %This simple result was one of the starting seeds for the development of Ramsey theory (see, for instance, \cite{Steele} for proofs and applications). 
 There are a number of different ways to generalise the Erd\H{o}s-Szekeres Theorem 
to higher dimensions (see, for example, \cite{BurkillMonotone,BurkillMatrix,Kalmanson,Kruskal,Linial,Morse,Siders,Szabo}). Perhaps the most natural approach was developed thirty years ago by Fishburn and Graham~\cite{Fishburn}.

A {\em $d$-dimensional array} is an
injective function $f$ from $A_1 \times \ldots \times A_d$ to $\mathbb{R}$ where $A_1, \ldots A_d$ are non-empty subsets of $\mathbb{Z}$; we say $f$ has {\em size}
$|A_1|\times\dots\times|A_d|$; if $|A_i|=n$ for each $i$, it will be convenient to say that $f$ has
{\em size} $[n]^d$.  
A multidimensional array is said to be \textit{monotone} if for each direction all the $1$-dimensional subarrays in that direction are increasing or decreasing. 
In other words, for every $i$, one of the following holds:
\begin{itemize}
    \item For every choice of $a_j$, $j\ne i$, the function $f(a_1,\dots,a_{i-1},x,a_{i+1},\dots,a_d)$ is increasing in $x$.
    \item For every choice of $a_j$, $j\ne i$, the function $f(a_1,\dots,a_{i-1},x,a_{i+1},\dots,a_d)$ is decreasing in $x$.
\end{itemize}
Let $M_d(n)$ be the smallest $N$ such that a $d$-dimensional array on $[N]^{d}$ contains a monotone $d$-dimensional subarray of size $[n]^d$.
%$A_1 \times \dots  \times A_d$, where $A_i\subset [N]$ and $|A_i|=n$, for all $i\in [d]$. 
Fishburn and Graham~\cite{Fishburn} showed that $M_d(n)$ exists but their upper bounds were a tower of height $d$. 
Recently, Buci\'c, Sudakov, and Tran~\cite{Sudakov} proved considerably better upper bounds on $M_d(n)$, showing doubly exponential bounds in $n^{d-1}$ for $2$ and $3$ dimensions, and triply exponential bounds in $4$ or higher dimensions. 
\begin{theorem}(Buci\'c, Sudakov, and Tran)\label{bst}
%\leavevmode
\begin{enumerate}
    \item[i)] $M_2(n)\leq 2^{2^{(2+o(1))n}}$,
    \item[ii)] $M_3(n) \leq 2^{2^{(2+o(1))n^2}}$,
    \item[iii)]  $M_d(n) \leq 2^{2^{2^{O_d(n^{d-1})}}}$, \text{ for } $d\geq 4$. 
\end{enumerate}
\end{theorem}

Buci\'c, Sudakov, and Tran~\cite{Sudakov} asked whether a better bound could be proved in four and higher dimensions, speculating that the triply exponential bound could be reduced to a doubly exponential bound.
Using the methods from our proof of \Cref{thm:mainRamsey}, we resolve their question, showing that a doubly exponential upper bound holds in all dimensions. 

% We confirm that this is the case.
\begin{theorem}\label{thm: mainES}
For every $d\geq 2$, there is $C_d>0$, such that for every positive $n$, $M_d(n)\leq 2^{n^{C_dn^{d-1}}}$.
\end{theorem}
%We point out that similar arguments as in \cite{Sudakov} would not give a fixed tower height upper bound in Theorem~\ref{thm:mainRamsey} since in their proof one needs to iterate Erd\H{o}s-Szekeres $n^{d-1}$ times while in the Ramsey setting that would only give a tower height of order $n^{d-1}$ as an upper bound. 
%In order to circumvent this issue one needs to come up with a different approach; first we find a \texit{large} and very structured sub-grid and then by a careful inductive argument, we show it must contain a big mochromatic grid. 

In our proof, instead of trying to find directly a monochromatic sub-product of cliques, we first find an intermediate structured object which we call \textit{$d$-consistent} and then we show that a density-type argument on these structured objects allows us to get a monochromatic $\ourbox^{d}\!K_n$. We then use this new approach to improve the upper bound in Theorem \ref{bst}, leading to \Cref{thm: mainES}.
This differs from the methods of \cite{Sudakov}, which makes significant use of the
fact that the Erd\H{o}s-Szekeres Theorem has a polynomial bound and so it can be iterated $n^{d-1}$ times to obtain the bounds for $d\ge4$ in Theorem~\ref{bst}. Ramsey's Theorem, however, has an exponential bound 
and so a different approach is required to avoid getting levels of exponentiation.

In our proof, we first find a new intermediate notion of a structured object which we call \textit{$d$-consistent} and then we show that a density-type argument on these structured objects allows us to get a monochromatic $\ourbox^{d}\!K_n$. We then use this new approach to improve the upper bound in Theorem \ref{bst}, leading to \Cref{thm: mainES}.

% \gal{We later use the same approach to improve the upper bound in the Erd\H{o}s-Szekeres result, proving \Cref{thm: mainES}.}

We finish the introduction by pointing out that our results also give improvements for multidimensional lexicographic-monotone array.  
In their paper, Fishburn and Graham \cite{Fishburn} introduced another natural generalisation for monotone sequences and the Erd\H{o}s-Szekeres theorem, which they called a \textit{lex-monotone array}.
A $d$-dimensional array $f$ is said to be \textit{lex-monotone}  if the following holds. There exists a permutation $\tau\in S_{[d]}$ and a sign vector $s\in \{-1,1\}^d$, such that $f(\textbf{x})<f(\textbf{y})$ if and only if the vector $(s_{\tau(i)}x_{\tau(i)})_{i\in [d]}$ is smaller in lexicographic ordering than vector $(s_{\tau(i)}y_{\tau(i)})_{i\in [d]}$, that is, if there is $i\in [d]$ such that $s_{\tau(j)}x_{\tau(j)}=s_{\tau(j)}y_{\tau(j)}$ for every $j<i$, and $s_{\tau(i)}x_{\tau(i)}<s_{\tau(i)}y_{\tau(i)}$.
Let $L_d(n)$ be the smallest $N$ such that a $d$-dimensional array on $[N]^{d}$ contains a lex-monotone $d$-dimensional subarray of size $[n]^d$.
Fishburn and Graham gave a tower bound of order $d-1$ on $L_d(n)$, which was improved by Buci\'c, Sudakov and Tran to a tower of order 5 for $d\geq 4$ (and triple exponential for $d=3$). We note that using \Cref{thm: mainES} together with Theorem 1.2 from \cite{Sudakov} we obtain a triple exponential upper bound for all $d\geq 3$, that is, $L_d(n)\leq 2^{2^{2^{C_dn^{d-2}}}}$. 
%but this will not give a double exponential bound in this case.

We prove Theorem \ref{thm:mainRamsey} in Section \ref{Sec:Rdn} and Theorem \ref{thm: mainES} in Section \ref{Sec:Mn}.  We conclude with some further discussion in Section \ref{sec:Conc}.

\section{Upper bound on $R_r(d,n)$}\label{Sec:Rdn}
To simplify notation we will identify the vertex set of $K_N$ with $[N]=\{1,\dots,N\}$ and the vertex set of $K_N\Osq \cdots \Osq K_N$ with $[N]^d$.  Suppose that, for each $i\in[d]$ we have a graph $G_i$ and a subgraph $H_i\subseteq G_i$. For $a\in V(G_d)$,  we write $H_1\Osq\dots\Osq H_{d-1}\Osq a$ for the copy of $H_1\Osq\dots\Osq H_{d-1}\subseteq G_1\Osq\dots\Osq G_{d}$ with vertex set $V(H_1)\times\dots\times V(H_{d-1})\times \{a\}$ (thus all vertices have $a$ as their $d$th coordinate; we will usually omit the braces around $\{a\}$ for simplicity). We define subgraphs such as $a_1\Osq \cdots\Osq a_{d-1}\Osq H_{d}$ analogously.
We will sometimes refer to induced subgraphs by their vertex sets: thus, for an edge-coloured graph $G$ we  say that $S\subset V(G)$ {\em contains a monochromatic $K_k$} if there is some set $T\subset S$ such that the induced subgraph $G[T]$ is a monochromatic copy of $K_k$.

We begin with a short proof for the case $d=2$, and then give a (more involved) argument for the general case $d\geq 3$.

\begin{proof}[Proof of Theorem~\ref{thm:mainRamsey} for $d=2$] Note that
$R_r(1,t)$ is just the usual Ramsey 
number for $K_t$, which is smaller than $r^{rt}$.
Let $N\coloneqq r^{r^{10rn^2}}$, and consider an $r$-colouring of the edges of $K_N\Osq K_N$.
Fix a set $S\subset [N]$ of size $r^{r^{3rn^2+1}}$. By Ramsey's Theorem, for each $i\in [N]$ we can find a monochromatic copy of $K_{r^{3rn^2}}$ in $S\times i$. Note that there are at most $r\cdot \binom{|S|}{ {r^{3rn^2}}} \leq r^{r^{7rn^2}}$ choices for the vertex set of each monochromatic copy of $K_{r^{3rn^2}}$ and its colour.
Since ${N}/{r^{r^{7rn^2}}}\geq r^{2rn}$, a pigeonhole argument shows that there is 
a set $A_1$ of size $r^{3rn^2}$ and a set $A_2\subset [N]$ of size at least $r^{2rn}$ such that, for each $a_2\in A_2$
the set $A_1\times a_2$ induces a monochromatic copy of $K_{r^{3rn^2}}$ (all with the same choice of colour). 
Applying Ramsey's Theorem once again, we can find in each set
$b\times A_2$ a monochromatic copy of $K_n$ (as 
$|A_2|=r^{2rn}$). As there are at most $r\cdot \binom{|A_2|}{n}\leq r^{2rn^2}$ choices for the vertex set and colour of this $K_n$, we can apply the pigeonhole principle again: there is a set 
$A'_1\subset A_1$ of size at least $|A_1|/r^{2rn^2}\geq k$ 
and a set $A_2'\subset A_2$ of size $n$ such that, for every $a_1\in A'_1$, the set $a_1\times A'_2$
forms a monochromatic copy of $K_n$ (and all with the same choice of colour). 
It is clear that $A'_1\times A'_2$ forms a copy of $K_n\Osq K_n$ that is monochromatic in both directions. Hence $R_r(2,n)\leq N$, as we wanted to show. 
\end{proof}

We now turn to the general argument for $d\ge3$.  Consider an 
$r$-edge-coloured product of complete graphs.  We say a $1$-dimensional $K_k$ is $1$-\textit{consistent} if it is monochromatic. For $d\geq 2$, we say $\ourbox^{d}\!K_k$ is $d$-\textit{consistent} if the following two conditions hold:
\begin{itemize}
\item for every $a\in [k]$, the subgraph $(\ourbox^{d-1}\!K_k)\Osq a$ has the same colouring: that is, for every edge $\textbf{x}\textbf{y}$ in $\ourbox^{d-1}\!K_k$, the colour of the edge between $\textbf{x}\times a$ and $\textbf{y}\times a$ is the same for every $a\in [k]$; and 
\item for some (and thus for every) $a$, the subgraph $\ourbox^{d-1}\!K_k\Osq a$ is $(d-1)$-\textit{consistent}. 
\end{itemize}
In other words, for each $i\in[d]$, the `$i$-dimensional subspaces' $(\ourbox^i K_N)\Osq a_{i+1}\Osq\cdots\Osq a_d$ have the same colouring for every choice of $a_{i+1},\dots,a_d$.

The proof splits into two lemmas. Given suitable $N\gg k\gg n$ and an $r$-edge colouring of $\ourbox^d\!K_N$, we first find a  $d$-\textit{consistent}  subgraph $H=\ourbox^{d}\!K_k$. We then show that $H$ contains a monochromatic copy of $\ourbox^{d}\!K_n$.

\begin{lemma}\label{lem: 1}
For every $d\ge1$ there is a constant $g(d)$ such that the following holds.  Let $r,k$ be positive integers and $0<\epsilon<1/2$. Suppose that 
$N\geq \epsilon^{-g(d)k^{d-1}}\cdot r^{g(d)rk^{d}}$ and $\ourbox^d K_N$ has a $d$-consistent $r$-edge-colouring. Then every set $S$ of at least $\epsilon N^{d}$ vertices contains a copy of $\ourbox^{d}\!K_k$ which is monochromatic in every direction. 
\end{lemma}

\begin{proof}
We argue by induction on $d$ that $g(d)=(12+2r)^{d-1}$ will do. It is clear that when $d=1$ it is sufficient to have $N\geq r^{rk}/\epsilon$, which we do as $g(1)=1$. So we assume that $d\geq 2$ and we have handled smaller cases. Let $T$ be the set of elements ${\mathbf v}\in[N]^{d-1}$ such that ${\mathbf v}\times[N]$ contains at least $\epsilon N/2$ elements of $S$. A counting argument shows that $|T|\ge \epsilon N^{d-1}/2$.
Let $A\subset[N]$ be a random subset of size $(10/\epsilon)r^{rk}$. For each ${\mathbf v}\in T$, with probability at least $2/3$ the set $\mathbf v\times A$ contains at least $r^{rk}$ elements of $S$. So we can choose $A$ so that the set
    $$T':=\{ {\mathbf v}\in T: |({\mathbf v}\times A)\cap S|\ge r^{rk}\}$$
has size at least $(2/3)|T|\ge(\epsilon/3)N^{d-1}$.
     By Ramsey's Theorem, for each ${\mathbf v}\in T'$, there is a monochromatic copy of $K_k$ contained in 
$({\mathbf v}\times A)\cap S$, 
say on vertices $B_{\mathbf v}\subset A$. There are at most $r\binom{(10/\epsilon)r^{rk}}{k}\leq \epsilon^{-9k}r^{rk^2+1}$ 
     choices for $B_{\mathbf v}$ and the colour of the corresponding copy of $K_k$, so there are $B\subset A$ and $U\subset T'$ such that
$$|U|\ge \epsilon^{9k}r^{-rk^2-1}|T'| \ge (\epsilon/3)\epsilon^{9k}r^{-rk^2-1}N^{d-1}$$
and $B_{\mathbf v}=B$ for every $\mathbf v\in U$.  
Now let $\tilde{\epsilon}=(\epsilon/3)\epsilon^{9k}r^{-rk^2-1}\ge \epsilon^{12k}r^{-rk^2-1}$, so $|U|\geq \tilde{\epsilon}N^{d-1}$. We now apply the inductive hypothesis to $[N]^{d-1}$, with the colouring inherited from $[N]^{d-1}\times a$ (which by consistency is the same for any $b\in B$) with the subset $U$ playing the role of $S$.
Since $N\geq \epsilon^{-g(d)k^{d-1}}r^{g(d)rk^d}
\geq \tilde{\epsilon}^{-g(d-1)k^{d-2}}\cdot r^{g(d-1)rk^{d-1}}$, 
we obtain a $(d-1)$-dimensional product $P$ of copies of $K_k$ which is monochromatic in every direction.  Then $P\times B$ gives a $d$-dimensional product of copies of $K_k$ which is monochromatic in every direction (as $P\times b$ is coloured in the same way for every $b\in B$, and all sets $\textbf{v}\times B$ ($\textbf{v}\in B)$ give monochromatic copies of $K_k$ with the same colour).
\end{proof}

\begin{lemma}\label{lem: 2}
For every positive integer $d$ there is $f(d)$ such that for every $r,k$ the following holds. Let $N\coloneqq r^{f(d)rk^{d}}$. Then in every $r$-colouring of $\ourbox^{d}\!K_N$ there is a $d$-consistent $\ourbox^{d}\!K_k$. 
\end{lemma}
\begin{proof}
We argue by induction on $d$ that $f(d)=2^{d-1}d!$ will do. For $d=1$, this is true by Ramsey's Theorem. 
Now, let $M\coloneqq r^{f(d-1)rk^{d-1}}$ and consider the subgraph $(\ourbox^{d-1}\!K_{M})\Osq K_N$. By induction any $r$-colouring of $\ourbox^{d-1}\!K_{M}$ contains a $(d-1)$-consistent $\ourbox^{d-1}\!K_k$. Therefore, for every $a\in [N]$ the subgraph $\ourbox^{d-1}\!K_{M}\times a$ contains a $(d-1)$-consistent $\ourbox^{d-1}\!K_k\times a $.
 As there are most $r^{(d-1)k^{2+(d-2)}}$ possible $r$-colourings of $\ourbox^{d-1}K_k$ and at most $\binom{M}{k}^{d-1}$ possible vertex sets, 
 there are at most $r^{(d-1)k^{2+(d-2)}}\binom{M}{k}^{d-1}\le r^{(d-1)k^d}M^{k(d-1)}/k< r^{f(d)rk^d}/k=N/k$ possible combinations.
Thus, by the pigeonhole principle, there is a $(d-1)$-consistent colour pattern $c$ of $\ourbox^{d-1}K_k$ and a set $A\subset [N]$ of size $k$, such that for every $a\in A$, the subgraph $\ourbox^{d-1}\!K_k \times a$ has colour pattern $c$, and all these boxes lie on the same vertex set in the first $d-1$ dimensions, as we wanted to show. 
\end{proof}

\begin{proof}[Proof of Theorem~\ref{thm:mainRamsey}]
We will show that the statement holds with $C_d=3dg(d-1)+f(d)+1$, where $f$ and $g$ are
any functions satisfying the previous two lemmas.
Let $N\coloneqq r^{r^{C_drn^{d}}}$. Let $t\coloneqq r^{3g(d-1)rn^{d}}$ and $u\coloneqq r^{rn}$. 
Applying Lemma~\ref{lem: 2} to $\ourbox^{d}\!K_{N}$, we obtain a $d$-consistent copy $B$ of $\ourbox^{d}\!K_t$. Relabelling, and restricting the final coordinate to $u$ choices, we may assume that we have a $d$-consistent colouring of $\ourbox^{d-1}\!K_t\Osq K_u$. 

For every $a_1\times \dots \times a_{d-1}$, where each $a_i\in [t]$, we can apply Ramsey's Theorem to $a_1\Osq \dots \Osq a_{d-1}\Osq K_{u}$ to get a monochromatic copy of $K_n$. 
There are $\binom{u}{n}\cdot r$ choices of colour and coordinates $C$ for this copy; setting
$\epsilon=1/(ru^{n})$ we see that there is some set $S$ of at least $\epsilon t^{d-1}$ vertices in $\ourbox^{d-1}\!K_t$ for which both colour and coordinates agree.

%We colour each $a_1\times \dots \times a_{d-1}$, %where $a_i\in [t] $ depending on the colour and %the vertex set of the monochromatic $a_1\times %\dots \times a_{d-1} \Osq K_n$. By a simple %counting argument there are at most 
%$\binom{t_1}{n}\cdot r$ many distinct colours and %hence there is a colour which appears in at least %$(r\cdot \binom{t_1}{n})^{-1}\coloneqq \epsilon'$-%proportion of the vertices of 
%$\ourbox^{d-1}\!K_t$, say set $S$.

Finally, we apply \Cref{lem: 1} to 
$S\subset\ourbox^{d-1}\!K_t$
and obtain a copy $K$ of $\ourbox^{d-1}\!K_n$ contained in $S$ that is monochromatic in every direction. This lemma holds for
$t\ge \epsilon^{-g(d-1)n^{d-2}}r^{g(d-1)rn^{d-1}}$, which holds by our choice of constants,
as 
$$\epsilon^{-g(d-1)n^{d-2}}
=(ru^n)^{g(d-1)n^{d-2}}
=r^{(rn^2+1)g(d-1)n^{d-2}}.$$
Then $K\Osq C$ is a copy of $\ourbox^{d}\!K_n$ that is monochromatic in every direction.
%, where $t$ plays the role of $N$, $n$ the role of $k$, $\epsilon'$ the role of $\epsilon$ ($d-1$ the role of $d$). Since $t\geq \epsilon'^{-f'(d-1)rn^{d-2}}\cdot r^{f'(d-1)rn^{d}}$, we must have a monochromatic $\ourbox^{d-1}\!K_n \subset \ourbox^{d-1}\!K_t$. Using the fact that $B$ is $d$-consistent, $\ourbox^{d-1}\!K_n\times a$ is monochromatic for every $a\in [t]$. By construction, as $\ourbox^{d-1}\!K_n\subset S$, we can stack these boxes along a set $A\subset [t]$ of size $n$ giving the required monochromatic $\ourbox^{d}\!K_n$.
\end{proof}

\section{Upper bound on \texorpdfstring{$M_d(n)$}{Md(n)}}\label{Sec:Mn}

We note first that an immediate application of \Cref{thm:mainRamsey} gives an upper bound of $M_d(n)\leq 2^{2^{C_dn^{d}}}$ by giving an edge $xy$ colour red if $f(x)<f(y)$ and blue otherwise. However, we wish to get a dependence of $d-1$ in the exponent in \Cref{thm: mainES}, and so prove it separately.

The proof follows along the same route as the proof of Theorem~\ref{thm:mainRamsey}. 
However, we need an modified version of Lemma~\ref{lem: 1} where the exponent is slightly better. 
In what follows, , we will assume all the arrays hereafter are injective, which we can ensure by a small perturbation to the values. 
We say a $1$-dimensional array is \textit{consistent} if it is monotone. For $d\geq 2$, we say an array $f:[N]^d\rightarrow \mathbb{R}$ is $d$-consistent if 
the following two conditions hold:
\begin{itemize}
\item for every $a\in [N]$, $f$ restricted to $[N]^{d-1}\times \{a\}$ has the same order pattern i.e. for every $\textbf{x}, \textbf{y}\in [N]^{d-1}$, if $f(\textbf{x}\times a)< f(\textbf{y}\times a)$, for some $a$ then the same holds for all $a\in [N]$; and 
\item for some (and thus for every) $a$, the $(d-1)$-dimensional subarray $f:[N]^{d-1}\times a\rightarrow \mathbb{R}$ is $(d-1)$-\textit{consistent}. 
\end{itemize}

As in Section 2, we will show that being $d$-consistent with the appropriate parameters is enough for our purpose.

\begin{lemma}\label{lem: 1*}
For every $d\geq 1$, there is a constant $g(d)$ such that the following holds.  Let $r,k$ be positive integers and $0<\epsilon<1/2$ . Suppose that $N\geq \epsilon^{-g(d)k^{d-1}}\cdot k^{g(d)k^{d-1}}$ and the array $f:[N]^d \rightarrow \mathbb{R}$ is $d$-consistent. Then every  $S\subseteq [N]^d$ of size at least $\epsilon N^{d}$ contains a monotone subarray of size $[k]^{d}$.
%
%Let $k,d$ be positive integers and $0<\epsilon<1/2$. Suppose $f:[N]^d \rightarrow \mathbb{R}$ is $d$-consistent. 
%Then, there is $f'(d)$ (just depending on $d$) every $S\subseteq [N]^d$ of size at least $\epsilon N^{d}$ contains a monotone subarray of size $k^{d}$, provided $N\geq \epsilon^{-f'(d)k^{d-1}}\cdot k^{f'(d)k^{d-1}}$. 
\end{lemma}
\begin{proof}
We follow the proof of Lemma~\ref{lem: 1}, with $k^2$ playing the role of $r^{rk}$, $g(d)=2\cdot15^{d-1}$, and with the remark that in the 1-dimensional case, the Erd\H{o}s-Szekeres Theorem implies that it is enough to have $N\geq \epsilon^{-1}(2k)^2$.
We include the full details here for completeness.

We argue by induction on $d$ that $g(d)=2\cdot15^{d-1}$ will do. It is clear that when $d=1$ it is sufficient to have $N\geq k^2/\epsilon$, which we do as $g(1)=2$. So we assume that $d\geq 2$ and we have handled smaller cases. Let $T$ be the set of elements ${\mathbf v}\in[N]^{d-1}$ such that ${\mathbf v}\times[N]$ contains at least $\epsilon N/2$ elements of $S$. A counting argument shows that $|T|\ge \epsilon N^{d-1}/2$.
Let $A\subset[N]$ be a random subset of size $(10/\epsilon)k^2$. For each ${\mathbf v}\in T$, with probability at least $2/3$ the set $\mathbf v\times A$ contains at least $k^2$ elements of $S$. So we can choose $A$ so that the set
    $$T':=\{ {\mathbf v}\in T: |({\mathbf v}\times A)\cap S|\ge k^2\}$$
has size at least $(2/3)|T|\ge(\epsilon/3)N^{d-1}$.
By the Erd\H{o}s-Szekeres Theorem, for each ${\mathbf v}\in T'$, there is a monotone subarray of size $k$ contained in 
$({\mathbf v}\times A)\cap S$, 
say on vertices $B_{\mathbf v}\subset A$. There are at most $2\binom{(10/\epsilon)k^2}k\leq \epsilon^{-9k}k^{2k}$
% \alex{9 could be 5 here, but maybe not worth changing. In fact I feel a Jimi Hendrix somg coming on: https://www.youtube.com/watch?v=vZuFq4CfRR8} \gal{I agree, also with the Jimi Hendrix part :)}
     choices for $B_{\mathbf v}$ and the direction of monotonicity (i.e.~increasing or decreasing), so there are $B\subset A$ and $U\subset T'$ such that
$$|U|\ge \epsilon^{9k}k^{-2k}|T'| \ge (\epsilon/3)\epsilon^{9k}k^{-2k}N^{d-1}$$
and $B_{\mathbf v}=B$ for every $\mathbf v\in U$.  
Now let $\tilde{\epsilon}=(\epsilon/3)\epsilon^{9k}k^{-2k}\ge \epsilon^{12k}k^{-2k}$, then $|U|\geq \tilde{\epsilon}N^{d-1}$. We now apply the inductive hypothesis to the subset $U\subseteq [N]^{d-1}$, with the values inherited from $[N]^{d-1}\times b$  (which by consistency are the same for any $b\in B$).
Since $$N\geq \epsilon^{-g(d)k^{d-1}}k^{g(d)k^{d-1}}
\geq \tilde{\epsilon}^{-g(d-1)k^{d-2}}\cdot k^{g(d-1)k^{d-2}},$$ 
we obtain a $(d-1)$-dimensional monotone subarray of size $[k]^{d-1}$.  Then $P\times B$ gives a monotone subarray of size $[k]^d$.
%$d$-dimensional product of monotone subarrays of size $k$ in every direction.
\end{proof}

It is thus enough to find a $d$-consistent array.
The  following lemma is analogue of \Cref{lem: 2} for arrays. 

\begin{lemma}\label{lem: 2*}
For every positive integer $d$ there is $f(d)$ such that for every $k$ the following holds. Let $N\coloneqq k^{f(d)k^{d-1}}$. Then in every $d$-dimensional array of size $[N]^d$  there is a $d$-consistent subarray of size $[k]^d$. 
\end{lemma}

\begin{proof}
We argue by induction on $d$ that $f(d)=2^{d-1}d!+1$ will do. For $d=1$, this is true by the Erd\H{o}s-Szekeres Theorem.
Now let $M\coloneqq k^{f(d-1)k^{d-2}}$. By induction any array on $[M]^{d-1}$ contains a $(d-1)$-consistent array of size $[k]^{d-1}$ . Therefore, for every $a\in [N]$ the $(d-1)$-dimensional array on $[M]^{d-1}\times \{a\}$ contains a $(d-1)$-consistent array on $A_1\times \dots\times A_{d-1}\times \{a\} $, where $|A_i|=n$ for all $i\in [d-1]$. 
 As there are most $k^{(d-1)k^{1+(d-2)}}$ possible orderings of an array of size $k^{d-1}$, and at most  $\binom{M}{k}^{d-1}$ 
choices for the sets $A_i$,
 %possible vertex sets where they can fit, 
 there are at most $k^{(d-1)k^{1+(d-2)}}\binom{M}{k}^{d-1}< k^{(d-1)k^{d-1}}M^{k(d-1)}/k\le k^{f(d)k^{d-1}}/k=N/k$ possible combinations. Thus, by the pigeonhole principle there are an ordering $\mathcal O$ of $[k]^{d-1}$ and a set $A\subset [N]$ of size $k$, where  $[k]^{d-1} \times \{a\}$ have the same ordering $\mathcal O$
 for every $a\in A$ and all these $(d-1)$-dimensional arrays lie on the same vertex set in the first $d-1$ dimensions, as we wanted to show. 
\end{proof}

\begin{proof}[Proof of \Cref{thm: mainES}]
We will show that the statement holds with $C_d=4dg(d-1)+f(d)$, where $f$ and $g$ are
any functions satisfying \Cref{lem: 1*} and \Cref{lem: 2*}. Let $N\coloneqq n^{n^{C_dn^{d-1}}}$. Let $t\coloneqq n^{3g(d-1)n^{d-1}}$ and $u\coloneqq n^{2}$. 
Applying Lemma~\ref{lem: 2*} to $[N]^d$, with $t$ playing the role of $k$, we obtain a $d$-consistent array on $B_1\times\dots\times B_d\coloneqq B$, where $|B_i|=t$ for all $i\in[d]$. Relabelling, and restricting the final coordinate to $u$ choices, we may assume that we have a $d$-consistent 
array on $[t]^{d-1}\times [u]$.
%colouring of
%$B_1\times\dots\times B_{d-1}\times [u]$ where $|B_i|=t$ for all $i\in[d-1]$. 
For every $a_1\times \dots \times a_{d-1}\in[t]^{d-1}$, we can apply the Erd\H{o}s-Szekeres Theorem to the 1-dimensional array on $a_1\times \dots \times a_{d-1}\times [u]$ to get a monotone subarray of size $n$. 
By a simple counting argument there are at most $\binom{u}{n}\cdot 2$ choices for the coordinates of the subarray and whether it is increasing or decreasing,
and hence there is some choice of these which occurs on a fraction of at least $(2\cdot \binom{u}{n})^{-1}\coloneqq \hat\epsilon$ of the vertices of $[t]^{d-1}$, say on a set $S$.  let $A\subseteq[u]$ be the common choice of coordinates for the monotonic subarrays corresponding to $S$.
%For every every $a_1\times \dots \times a_{d-1}$, where $a_i\in B_i $, we can apply Erd\H{o}s-Szekeres' Theorem to the array on $\{a_1\}\times \dots \times \{a_{d-1}\}\times [u]$ to get a monotone subarray of size $n$. We colour each $\{a_1\}\times \dots \times \{a_{d-1}\}$, where $a_i\in B_i$ depending whether the monotone array whose first $d-1$ coordinates are $\{a_1\}\times \dots \times \{a_{d-1}\}$ is increasing or decreasing and where the vertex set of the array lies. By a simple counting argument there are at most $\binom{u}{n}\cdot 2$ many distinct colours and hence there is a colour which appears in at least $(2\cdot \binom{u}{n})^{-1}\coloneqq \hat\epsilon$-proportion of the vertices of $B_1\times\dots\times B_{d-1}$, say set $S$.

Finally, we apply \Cref{lem: 1*} to $S\subseteq[t]^{d-1}$. By the choices of $t,u,\hat\epsilon$ we have that $t\geq \hat\epsilon^{-g(d-1)n^{d-2}}\cdot n^{g(d-1)n^{d-2}}$, and so we obtain a monotone $(d-1)$-dimensional array $T$ of size $[u]^{d-1}$. Since $B$ is $d$-consistent, $T\times \{a\}$ is monotone for every $a\in A$ (with the same choice of direction of monotonicity). By construction, this gives a monotone $d$-dimensional array on $T\times A$.
%Finally, we apply \Cref{lem: 1*}, where $t$ plays the role of $N$, $n$ the role of $k$, $\hat\epsilon$ the role of $\epsilon$ ($d-1$ the role of $d$). By the choices of $t,u,\hat\epsilon$ we have that $t\geq \hat\epsilon^{-g(d-1)n^{d-2}}\cdot n^{g(d-1)n^{d-2}}$, therefore we must have a monotone $(d-1)$-dimensional array $T$ of size $n$ in $B_1\times\dots\times B_{d-1}$. Using the property that $B$ is $d$-consistent, $T\times \{a\}$ is monotone for every $a\in B'_d$. By construction, as $T\subset S$, we can stack these arrays along a set $A\subset [u]$ of size $n$ giving the required monotone $d$-dimensional array $T\times A$.
\end{proof}

\section{Concluding remarks}\label{sec:Conc}
In Theorem \ref{thm:mainRamsey}, we have given a doubly exponential upper bound on the $d$-dimensional Ramsey numbers.  From below, we have only a singly exponential bound (which follows easily by considering random colourings).  It would be very interesting to close the gap.  In particular, it would be good to know whether there is a simple exponential upper bound, or whether the numbers grow more quickly.

\begin{problem}
Fix $r,d\ge2$.  Is $R_2(2,n)$ superexponential in $n$?  
\end{problem}

The same gap between lower and upper bounds is seen in the multimensional Erd\H{o}s-Szekeres Theorem.  
Buci\'c, Sudakov and Tran \cite{Sudakov} gave a doubly exponential upper bound for $d=2,3$; and 
\Cref{thm: mainES} gives a doubly exponential upper bound for $d\ge4$ (improving on the previous triply exponential upper bound \cite{Sudakov}).  But a frustrating gap between singly and doubly exponential bounds remains.  

It is easy to see that $M_2(n)\leq R_2(2,n)$ for every $n$. It would be interesting to see if $R_2(2, n)$ can be bounded above by a function of $M_2(n)$.
For example, is the following true. 
\begin{problem}
Does there exist $C>0$ such that $\log R_2(2,n)\leq (\log(M_2(n))^C$?
% Is $\log R_2(2,n)$ bounded from above by some polynomial in $\log(M_2(n))$?
\end{problem}

\section{Acknowledgements}
The authors would like to thank Zachary Hunter and Matija Buci\'c for their helpful comments. We also thank the anonymous referee for their suggestions.

%%% AUTHOR:
%%% Bibliography goes here. Note that the arXiv cannot process bibtex
%%% or biber bibliographies.  Example of acceptable bibliograpy format:
\bibliographystyle{amsplain}

%% AUTHOR: You can generate such a bibliography from a .bib file by 
%% running pdflatex/bibtex/pdflatex/pdflatex and then pasting the .bbl file
%% between \begin{thebibliography} and \end{bibliography}

%%% AUTHOR: Include a short description of each author following the
%%% structure below. Use the same short tags used previously.  
%%% Use \imageat{} and \imagedot{} instead of "@" and "." in
%%% email addresses-this replaces the symbols with graphics to avoid 
%%% e-mail address harvesting from the .pdf file
\begin{dajauthors}
\begin{authorinfo}[antonio]
Ant\'onio Gir\~ao\\
Mathematical Institute\\
University of Oxford\\
Oxford, United Kingdom\\
  girao\imageat{}maths\imagedot{}ox\imagedot ac \imagedot uk 
\end{authorinfo}
\begin{authorinfo}[gal]
 Gal Kronenberg\\
Mathematical Institute\\
University of Oxford\\
Oxford, United Kingdom\\
  kronenberg\imageat{}maths\imagedot{}ox\imagedot ac \imagedot uk 
\end{authorinfo}
\begin{authorinfo}[alex]
 Alex Scott\\
Mathematical Institute\\
University of Oxford\\
Oxford, United Kingdom\\
  kronenberg\imageat{}maths\imagedot{}ox\imagedot ac \imagedot uk 
\end{authorinfo}

\end{dajauthors}

\end{document}